\documentclass[12pt, a4paper, oneside]{amsart}
\usepackage{amsmath, amsthm, amssymb, geometry, mathrsfs, hyperref, tikz-cd}
\usepackage[shortlabels, inline]{enumitem}

\newtheorem{theorem}{Theorem}[section]

\newtheorem{lemma}[theorem]{Lemma}
\newtheorem{corollary}[theorem]{Corollary}

\theoremstyle{definition}
\newtheorem{definition}[theorem]{Definition}

\newtheorem{example}[theorem]{Example}
\newtheorem{remark}[theorem]{Remark}

\newcommand\C{\mathbb{C}}
\newcommand\D{\mathbb{D}}

\newcommand\N{\mathbb{N}}

\newcommand{\cA}{\mathcal{A}}

\newcommand{\cO}{\mathcal{O}}

\newcommand{\cV}{\mathcal{V}}

\newcommand{\der}{\mathrm{d}}

\newcommand{\id}{\mathrm{id}}

\newcommand{\T}{\mathrm{T}}

\DeclareMathOperator{\Bl}{Bl}
\DeclareMathOperator{\Op}{Op}
\DeclareMathOperator{\Supp}{Supp}
\DeclareMathOperator{\Sing}{Sing}

\newcommand{\CAP}{\mathrm{CAP}}

\newcommand{\BOPA}{\mathrm{BOPA}}
\newcommand{\BOPI}{\mathrm{BOPI}}

\newcommand{\BOPAI}{\mathrm{BOPAI}}

\newcommand{\POPA}{\mathrm{POPA}}
\newcommand{\POPI}{\mathrm{POPI}}

\newcommand{\Ell}{\mathrm{Ell}}

\newcommand{\aCAP}{\mathrm{aCAP}}

\begin{document}

\title[An implicit function theorem for sprays and applications]{An implicit function theorem for sprays \\ and applications to Oka theory}
\author{Yuta Kusakabe}
\address{Department of Mathematics, Graduate School of Science, Osaka University, Toyonaka, Osaka 560-0043, Japan}
\email{y-kusakabe@cr.math.sci.osaka-u.ac.jp}

\begin{abstract}
We solve fundamental problems in Oka theory by establishing an implicit function theorem for sprays. 
As the first application of our implicit function theorem, we obtain an elementary proof of the fact that approximation yields interpolation.
This proof and L\'{a}russon's elementary proof of the converse give an elementary proof of the equivalence between approximation and interpolation.
The second application concerns the Oka property of a blowup.
We prove that the blowup of an algebraically Oka manifold along a smooth algebraic center is Oka.
In the appendix, equivariantly Oka manifolds are characterized by the equivariant version of Gromov's condition $\Ell_{1}$, and the equivariant localization principle is also given.
\end{abstract}

\maketitle

%
%

\section{Introduction}

In the present paper, we establish an implicit function theorem for \emph{(holomorphic) sprays} to solve fundamental problems in modern Oka theory initiated by Gromov \cite{Gromov1989} in 1989 (cf. \cite{Forstneric2013,Forstneric2017}).
Here is the definition of sprays.

\begin{definition}
\label{definition:spray}
Let $X$ be a complex space and $Y$ be a complex manifold.
\begin{enumerate}[leftmargin=*]
\item A \emph{(local) spray} over a holomorphic map $f:X\to Y$ is a holomorphic map $s:U\to Y$ from an open neighborhood of the zero section of a holomorphic vector bundle $E\to X$ such that $s(0_{x})=f(x)$ for all $x\in X$.
Particularly in the case of $U=E$, $s$ is also called a \emph{global spray}.
\item A spray $s:U\to Y$ is \emph{dominating} if $s|_{U_{x}}:U_{x}\to Y$ is a submersion at $0_{x}$ for each $x\in X$.
\end{enumerate}
\end{definition}

The following is our implicit function theorem for sprays.
Here, the space of holomorphic maps $\cO(U,Y)$ and the space of holomorphic sections $\Gamma(\Omega,U)$ are equipped with the compact-open topologies.

\begin{theorem}
\label{theorem:implicit}
Let $X$ be a Stein space, $\Omega\Subset X$ be a relatively compact open subset, $Y$ be a complex manifold and $s_{f}:U\to Y$ be a dominating spray over a holomorphic map $f:X\to Y$.
Then there exist an open neighborhood $\cV\subset\cO(U,Y)$ of $s_{f}$ and a continuous map $\varphi:\cV\to\Gamma(\Omega,U)$ such that $\varphi(s_{f})=0$ and $s\circ\varphi(s)=f|_{\Omega}$ for all $s\in\cV$.
\end{theorem}

The proof of Theorem \ref{theorem:implicit} is given in Section \ref{section:proof}.
The relative version of this theorem (Remark \ref{remark:implicit}) was already used implicitly in our previous paper to prove the implication from CAP to convex ellipticity \cite[Proposition 5.3]{Kusakabe}.
In this paper, we give two more applications to Oka theory.
To state the first application, let us recall the following Oka properties.
Here, $\Op K$ denotes a non-specified open neighborhood of $K$.

\begin{definition}
\label{definition:BOP}
Let $Y$ be a complex manifold.
\begin{enumerate}[leftmargin=*]
\item $Y$ enjoys the \emph{Basic Oka Property with Approximation (BOPA)} if for any Stein space $X$, any compact $\cO(X)$-convex subset $K\subset X$ and any continuous map $f_{0}:X\to Y$ such that $f_{0}|_{\Op K}$ is holomorphic, there exists a homotopy $f_{t}:X\to Y$ $(t\in[0,1])$ such that $f_{t}|_{\Op K}$ is holomorphic and approximates $f_{0}$ uniformly on $K$ for each $t\in[0,1]$, and $f_{1}:X\to Y$ is holomorphic.
\item $Y$ enjoys the \emph{Basic Oka Property with Interpolation (BOPI)} if for any Stein space $X$, any closed complex subvariety $X'\subset X$ and any continuous map $f_{0}:X\to Y$ such that $f_{0}|_{X'}$ is holomorphic, there exists a homotopy $f_{t}:X\to Y$ $(t\in[0,1])$ such that $f_{t}|_{X'}=f_{0}|_{X'}$ for each $t\in[0,1]$, and $f_{1}:X\to Y$ is holomorphic.
\end{enumerate}
\end{definition}

It is one of the results in L\'arusson's paper \cite{Larusson2005} that BOPI implies BOPA.
The converse is also true by Forstneri\v{c}'s Oka principle (see \cite[\S5.15]{Forstneric2017}), and a complex manifold satisfying these equivalent conditions is called an \emph{Oka manifold}.
Compared with L\'{a}russon's proof, however, the proof of Forstneri\v{c}'s Oka principle is not easy.
In Section \ref{section:BOPA=>BOPI}, we give an elementary proof of the following fact by using Theorem \ref{theorem:implicit}.

\begin{corollary}
\label{corollary:BOPA=>BOPI}
For any complex manifold, $\BOPA$ implies $\BOPI$.
\end{corollary}

The same argument gives an elementary proof of the implication from POPA to POPI (Corollary \ref{corollary:POPA=>POPI}).
Recently, Kutzschebauch, L\'{a}russon and Schwarz \cite{Kutzschebauchc} introduced the equivariant version of Oka properties.
Our elementary proof is applicable also to the equivariant setting (see Remark \ref{remark:generalization}).

The second application concerns the Oka property of a blowup.
In Oka theory, it is important to understand when the blowup of an Oka manifold along a smooth center is Oka.
It is known that there exists a discrete set $D\subset\C^{2}$ such that the blowup $\Bl_{D}\C^{2}$ of $\C^{2}$ along $D$ is not Oka \cite[Example A.3]{Kusakabe2019}.
On the other hand, L\'{a}russon and Truong proved that the blowup of a manifold of Class $\cA$ (e.g. $\C^{n}$) along a smooth algebraic center is Oka \cite[Main Theorem]{Larusson2017} (see also \cite{Kaliman2018}).
A manifold of Class $\cA$ is known to be algebraically Oka.
To state the definition of algebraically Oka manifolds, recall that a complex manifold $Y$ is Oka if and only if $Y$ satisfies Gromov's condition $\Ell_{1}$, i.e. for any holomorphic map $f:X\to Y$ from a Stein manifold there exists a dominating global spray over $f$ \cite{Kusakabe2019,Kusakabe}.
Analogously, an \emph{algebraically Oka manifold} is defined to be an algebraic manifold $Y$ satisfying the algebraic version of $\Ell_{1}$, i.e. for any regular map $f:X\to Y$ from an affine manifold there exists an algebraic dominating global spray over $f$ (cf. \cite{Larusson2019}).
In Section \ref{section:blowup}, we generalize the result of L\'{a}russon and Truong as follows.

\begin{corollary}
\label{corollary:blowup}
Let $Y$ be an algebraically Oka manifold and $A\subset Y$ be a closed algebraic submanifold.
Then the blowup $\Bl_{A}Y$ of $Y$ along $A$ is Oka.
\end{corollary}

In the appendix, a few results on equivariant Oka theory are given.
More precisely, we give the equivariant versions of the characterization of Oka manifolds by $\Ell_{1}$ and the localization principle for Oka manifolds (Corollary \ref{corollary:characterization} and Theorem \ref{theorem:equivariant_localization}).
As an application of the latter, we prove that every smooth toric variety with its torus action is equivariantly Oka (Example \ref{example:toric}).

%
%

\section{Proof of Theorem \ref{theorem:implicit}}
\label{section:proof}

In order to prove Theorem \ref{theorem:implicit}, let us recall Rouch\'{e}'s theorem (see \cite[\S 10.2]{Chirka1989} for the definition of the multiplicity $\mu_{a}(f)$ of a holomorphic map).

\begin{theorem}[{cf. \cite[p.\,110]{Chirka1989}}]
\label{theorem:rouche}
Let $U\Subset\C^{N}$ be a bounded domain, $f,g:\overline U\to\C^{N}$ be continuous maps which are holomorphic on $U$ and satisfy $|f-g|<|f|$  on $\partial U$.
Then the numbers of zeroes of $f$ and $g$ in $U$ (counted with multiplicities) are equal:
\begin{align*}
\sum_{f(a)=0}\mu_{a}(f|_{U})=\sum_{g(b)=0}\mu_{b}(g|_{U}).
\end{align*}
\end{theorem}

\begin{proof}[Proof of Theorem \ref{theorem:implicit}]
By the definition of a spray, $U$ is an open neighborhood of the zero section of a holomorphic vector bundle $p:E\to X$.
Let $E'$ denote the holomorphic vector subbundle of $E$ with fibers
\begin{align*}
E'_x=\ker(\der(s_{f}|_{U_{x}})_{0_{x}}:E_{x}=\T_{0_{x}}U_{x}\to\T_{f(x)}Y),\ x\in X.
\end{align*}
Since $X$ is Stein, there exists another holomorphic vector subbundle $E''$ of $E$ such that $E=E'\oplus E''$ (cf. \cite[Corollary 2.6.6]{Forstneric2017}).
Then the map $(p,s_{f})|_{U\cap E''}:U\cap E''\to X\times Y$ restricts to a biholomorphic map from an open neighborhood $U''\subset U\cap E''$ of the zero section onto an open neighborhood of the graph of $f$.
Since $\Omega\Subset X$ is relatively compact, we may assume that $U''|_{\Omega}\subset U$ is also relatively compact.
Note that if a spray $s:U\to Y$ is sufficiently close to $s_{f}$, the restriction $s|_{U''|_{\Omega}}$ is uniformly close to $s_{f}|_{U''|_{\Omega}}$.
Then by using Rouch\'{e}'s theorem (Theorem \ref{theorem:rouche}) on each fiber of $p|_{U''|_{\Omega}}:U''|_{\Omega}\to\Omega$, we can conclude that there exists a unique holomorphic section $\varphi(s):\Omega\to U''$ which is close to the zero section and satisfies $s\circ\varphi(s)=f|_{\Omega}$.
By construction, $\varphi(s)$ depends on $s$ continuously and satisfies $\varphi(s_{f})=0$.
\end{proof}

\begin{remark}
\label{remark:implicit}
The same proof also gives the relative version (for relative sprays \cite[Definition 2.1]{Kusakabe}) and the parametric version of Theorem \ref{theorem:implicit} (use the arguments in \cite[p.\,254]{Forstneric2017}).
\end{remark}

%
%

\section{Elementary proof of that approximation implies interpolation}
\label{section:BOPA=>BOPI}

Let us recall the following fact which ensures the existence of dominating local sprays.
It is an easy application of Siu's theorem \cite{Siu1976} and Cartan's Theorem A.
Here, a compact subset of a complex space is called a \emph{Stein compact} if it admits a basis of open Stein neighborhoods.

\begin{lemma}[{cf. \cite[Lemma 5.10.4]{Forstneric2017}}]
\label{lemma:local_spray}
Let $K$ be a Stein compact in a complex space and $Y$ be a complex manifold.
Then for any holomorphic map $f:\Op K\to Y$ there exist an open neighborhood $W\subset\C^{N}$ of $0$ and a dominating spray $\Op K\times W\to Y$ over $f$.
\end{lemma}

\begin{proof}[Proof of Corollary \ref{corollary:BOPA=>BOPI}]
Assume that a complex manifold $Y$ enjoys $\BOPA$.
Let $X$, $X'$ and $f_{0}$ be as in the definition of BOPI (Definition \ref{definition:BOP}).
Since $X$ is Stein, there exists an exhausting sequence
\begin{align*}
\emptyset=K_0\subset K_1\subset K_2\subset\cdots\subset\bigcup_{j\in\N} K_j=X,\quad K_j\subset K_{j+1}^\circ
\end{align*}
of compact $\cO(X)$-convex subsets.
Set $t_j=1-2^{-j}$ for each $j\in\N\cup\{0\}$, and take a complete distance function $d$ on $Y$ and a positive number $\varepsilon>0$.
For each $j\in\N$, we shall inductively construct homotopy $f_{t}:X\to Y$ $(t\in[t_{j-1},t_{j}])$ such that $f_{t_{j}}|_{\Op K_{j}}$ is holomorphic and the following hold for each $t\in[t_{j-1},t_{j}]$;
\begin{itemize}
	\item $f_{t}|_{X'}=f_{0}|_{X'}$,
	\item $f_{t}|_{\Op K_{j-1}}$ is holomorphic and
	\begin{align*}
		\sup_{x\in K_{j-1}}d(f_{t}(x), f_{t_{j-1}}(x))<\frac{\varepsilon}{2^{j}}.
	\end{align*}
\end{itemize}
Then $f_1=\lim_{t\to 1}f_t:X\to Y$ exists and the homotopy $f_t$ $(t\in[0,1])$ has the desired properties.
Note that in each induction step it suffices to construct $f_{t}:\Op K_{j}\to Y$ $(t\in[t_{j-1},t_{j}])$ with the above properties since we can extend it as $f_{t_{j-1}+(t-t_{j-1})\chi(x)}(x)$ by using a continuous function $\chi:X\to[0,1]$ such that $\chi|_{\Op K_{j}}\equiv 1$ and $\Supp\chi$ is contained in a small open neighborhood of $K_{j}$.

Assume that we already have $f_{t}:X\to Y$ $(t\in[t_{j-1},t_{j}])$ for some $j\in\N\cup\{0\}$ where $t_{-1}=0$ and $K_{-1}=\emptyset$.
Set $L_{j}=K_j\cup(X'\cap K_{j+1})$ and note that it is $\cO(X)$-convex.
By using the semiglobal holomorphic extension theorem (cf. \cite[Theorem 3.4.1]{Forstneric2017}), after deforming $f_{t_j}$ if necessary, we may assume that $f_{t_j}|_{\Op L_{j}}$ is holomorphic.
Then by Lemma \ref{lemma:local_spray} there exists a dominating local spray $s_{t_{j}}:\Op L_{j}\times\Op\overline{\D^{N}}\to Y$ over $f_{t_{j}}|_{\Op L_{j}}$.
By using a retraction of $\C^N$ onto $\overline{\D^N}$ and a continuous function $\chi':X\to[0,1]$ such that $\chi'|_{\Op L_{j}}\equiv 1$ and $\Supp\chi'$ is contained in a small open neighborhood of $L_{j}$, we may assume that $s_{t_{j}}$ is extended to $X\times\C^N\to Y$ continuously so that $s_{t_{j}}(\cdot,0)=f_{t_{j}}$.
By BOPA of $Y$, there exists a homotopy $s_{t}:\Op K_{j+1}\times\C^N\to Y$ $(t\in[t_{j},t_{j+1}])$ such that $s_{t}|_{\Op L_{j}\times\Op\overline{\D^N}}$ is holomorphic and approximates $s_{t_{j}}$ on $\Op L_{j}\times\Op\overline{\D^N}$ for each $t\in[t_{j},t_{j+1}]$, and $s_{t_{j+1}}$ is holomorphic.
Then Theorem \ref{theorem:implicit} implies the existence of a homotopy of holomorphic maps $\varphi_{t}:\Op L_{j}\to\C^{N}$ $(t\in[t_{j},t_{j+1}])$ such that $\varphi_{t_{j}}\equiv0$ and $s_{t}\circ(\id,\varphi_{t})=f_{t_{j}}|_{\Op L_{j}}$ for each $t\in[t_{j},t_{j+1}]$.
We can extend this homotopy to $\varphi_{t}:\Op K_{j+1}\to\C^N$ continuously so that $\varphi_{t_{j}}\equiv0$.
Then by the parametric Cartan--Oka--Weil theorem (cf. \cite[Theorem 2.8.4]{Forstneric2017}), there exists a homotopy of holomorphic maps $\widetilde\varphi_{t}:\Op K_{j+1}\to\C^N$ $(t\in[t_{j},t_{j+1}])$ such that $\widetilde\varphi_{t_{j}}\equiv0$ and the following hold for each $t\in[t_{j},t_{j+1}]$;
\begin{itemize}
\item $\widetilde\varphi_{t}|_{X'\cap\Op K_{j+1}}=\varphi_{t}|_{X'\cap\Op K_{j+1}}$, and
\item $\widetilde\varphi_{t}|_{\Op K_{j}}$ approximates $\varphi_{t}$ uniformly on $K_{j}$.
\end{itemize}
Then the homotopy $f_t=s_t\circ(\id,\widetilde\varphi_t):\Op K_{j+1}\to Y$ $(t\in[t_j,t_{j+1}])$ has the desired properties.
\end{proof}

\begin{remark}
\label{remark:generalization}
(1) By taking a nontrivial compact $\cO(X)$-convex subset $K_{0}\subset X$ and a small number $\varepsilon>0$, we can see that the above argument also gives an elementary proof of that $\BOPA$ implies $\BOPAI$ (see \cite[\S5.15]{Forstneric2017} for the definition of $\BOPAI$).
\\
(2) The above proof is applicable also to the equivariant setting (see \cite{Kutzschebauchc} for the definition of the equivariant Oka properties).
Namely, it gives an elementary proof of that $G$-BOPA implies $G$-BOPI for a finite group $G$ (cf. \cite[Corollary 4.2]{Kutzschebauchc}).
\\
(3) L\'{a}russon's proof of that interpolation implies approximation \cite{Larusson2005} works even if $Y$ is a singular complex space (see \cite[Proposition 1.5]{Larkang2016}).
However, our proof of the converse is not applicable since we do not have a good definition of dominating sprays for singular targets.
\end{remark}

It is known that the parametric Oka properties $\POPA$ and $\POPI$ are also equivalent (see \cite[\S5.15]{Forstneric2017}).
The implication from $\POPI$ to $\POPA$ has an elementary proof due to L\'{a}russon \cite{Larusson2005}.
The parametric version of Theorem \ref{theorem:implicit} (Remark \ref{remark:implicit}) and the above argument yield an elementary proof of the converse.
Since the proof is almost the same, we omit the proof here.

\begin{corollary}
\label{corollary:POPA=>POPI}
For any complex manifold, $\POPA$ implies $\POPI$.
\end{corollary}

The Oka properties are generalized to the relative setting and it is known that approximation and interpolation are equivalent also for a holomorphic submersion (see \cite[\S 5.1]{Kusakabe}).
The straightforward generalization of our argument gives an elementary proof also in this relative setting (see Remark \ref{remark:implicit}).

%
%

\section{Blowups of algebraically Oka manifolds}
\label{section:blowup}

We first recall the Convex Approximation Property (CAP) and its algebraic version aCAP.

\begin{definition}
\label{definition:CAP}
A complex manifold (resp. an algebraic manifold) $Y$ enjoys \emph{CAP} (resp. \emph{aCAP}) if any holomorphic map from an open neighborhood of a compact convex set $K\subset\C^{n}$ $(n\in\N)$ to $Y$ can be uniformly approximated on $K$ by holomorphic maps (resp. regular maps) $\C^{n}\to Y$.
\end{definition}

The following facts relate the above properties to the Oka properties.

\begin{theorem}[{cf. \cite[\S 5.15 and Corollary 6.15.2]{Forstneric2017}, \cite[Theorem 1]{Larusson2019}}]
\label{theorem:CAP}
\ \\
(1) A complex manifold enjoys $\CAP$ if and only if it is Oka.
\\
(2) An algebraic manifold enjoys $\aCAP$ if it is algebraically Oka.
\end{theorem}

Let us prove the stability of $\aCAP$ under blowups which implies Corollary \ref{corollary:blowup} (compare with the instability of CAP \cite[Example A.3]{Kusakabe2019}).

\begin{corollary}
\label{corollary:aCAP}
Let $Y$ be an algebraic manifold and $A\subset Y$ be a closed algebraic submanifold.
If $Y$ enjoys $\aCAP$, then the blowup $\Bl_A Y$ also enjoys $\aCAP$.
In particular, the blowup $\Bl_{A}Y$ is Oka.
\end{corollary}

\begin{proof}
Let $\pi:\Bl_A Y\to Y$ denote the blowup.
Take a compact convex set $K\subset\C^{n}$ and a holomorphic map $f:\Op K\to\Bl_{A}Y$.
By Lemma \ref{lemma:local_spray}, there exists a dominating local spray $s':\Op K\times\Op\overline{\D^{N}}\to Y$ over $\pi\circ f$.
Since $Y$ enjoys aCAP, there exists a regular map $s:\C^{n}\times\C^{N}\to Y$ which approximates $s'$ on $\Op K\times\Op\overline{\D^{N}}$.
Then Theorem \ref{theorem:implicit} implies the existence of a holomorphic section $\varphi:\Op K\to\Op K\times\C^{N}$ of the trivial bundle which is close to the zero section and satisfies $s\circ\varphi=\pi\circ f$.
Let $\Sing(s)$ denote the singular locus of $s$ and $\Sigma\subset\C^{n+N}$ be the intersection of $\Sing(s)$ and the non-Cartier locus of the scheme theoretic inverse image $s^{-1}(A)$.
Then $\Sigma\subset\C^{n+N}$ is a closed algebraic subvariety of codimension at least two because $\C^{n+N}$ is smooth.
Since $s|_{\Op K\times\C^{N}}$ is dominating and $\varphi$ is close to the zero section, we may assume that $\varphi(K)\subset\C^{n+N}\setminus\Sing(s)\subset\C^{n+N}\setminus\Sigma$.
Note that the pullback $s^{*}\Bl_{A}Y\to\C^{n+N}\setminus\Sigma$ of $\Bl_{A}Y\to Y$ along the restriction $s:\C^{n+N}\setminus\Sigma\to Y$ coincides with the blowup $\Bl_{s^{-1}(A)\setminus\Sigma}(\C^{n+N}\setminus\Sigma)\to\C^{n+N}\setminus\Sigma$ over $\C^{n+N}\setminus\Sing(s)$.
By the universality of the pullback, there exists a holomorphic map $\widetilde\varphi:\Op K\to\Bl_{s^{-1}(A)\setminus\Sigma}(\C^{n+N}\setminus\Sigma)$ such that the following diagram commutes:
\begin{center}
\begin{tikzcd}
\Op K
\arrow[drr, bend left=15, "f"]
\arrow[ddr, bend right=20, "\varphi"']
\arrow[dr, dashed, "\widetilde\varphi"] & & \\
& \Bl_{s^{-1}(A)\setminus\Sigma}(\C^{n+N}\setminus\Sigma) \arrow[r, "\pi^{*}s"] \arrow[d, "s^{*}\pi"] & \Bl_{A}Y \arrow[d, "\pi"] \\
& \C^{n+N}\setminus\Sigma \arrow[r, "s"] & Y
\end{tikzcd}
\end{center}
By definition, if $s^{-1}(A)\setminus\Sigma$ is non-Cartier at $z\in s^{-1}(A)\setminus\Sigma$ then $s^{-1}(A)\setminus\Sigma$ is smooth at $z$.
Thus the result of L\'{a}russon and Truong \cite[Theorem 1]{Larusson2017} implies that the blowup $\Bl_{s^{-1}(A)\setminus\Sigma}(\C^{n+N}\setminus\Sigma)$ is algebraically Oka.
By Theorem \ref{theorem:CAP}, there exists a regular map $\tilde f:\C^n\to\Bl_{s^{-1}(A)\setminus\Sigma}(\C^{n+N}\setminus\Sigma)$ which approximates $\widetilde\varphi$ uniformly on $K$.
Then the regular map $\pi^{*}s\circ\tilde f:\C^{n}\to\Bl_{A}Y$ approximates $f$ on $K$.
\end{proof}

\begin{remark}
In his seminal paper \cite{Gromov1989}, Gromov asked whether the algebraic Oka property is a birational invariant \cite[Remark 3.5.E$'''$]{Gromov1989} and this problem is still open.
Since Corollary \ref{corollary:aCAP} and the converse of (2) in Theorem \ref{theorem:CAP} would imply the stability of the algebraic Oka property under blowups, it is natural to ask whether the converse of (2) in Theorem \ref{theorem:CAP} holds.
It will be studied in future work.
\end{remark}

\appendix

%
%

\section{Elliptic characterization and localization of equivariantly Oka manifolds}

As we mentioned in the introduction, Oka manifolds (resp. algebraically Oka manifolds) are characterized by $\Ell_{1}$ (resp. the algebraic version of $\Ell_{1}$).
These characterizations imply the following localization principles.

\begin{theorem}[{cf. \cite[Theorem 1.4]{Kusakabe2019} and \cite[Proposition 6.4.2]{Forstneric2017}}]
\label{theorem:localization}
Let $Y$ be a complex manifold (resp. an algebraic manifold).
Assume that each point of $Y$ has an open Oka (resp. algebraically Oka) neighborhood with respect to the analytic (resp. algebraic) Zariski topology\footnote{A subset of $Y$ is open with respect to the analytic (resp. algebraic) Zariski topology if its complement is a closed complex subvariety (resp. a closed algebraic subvariety).}.
Then $Y$ is Oka (resp. algebraically Oka).
\end{theorem}

In this appendix, we give the equivariant versions of these facts.
To this end, we need the following definitions.

\begin{definition}
Let $G$ be a reductive complex Lie group and $Y$ be a $G$-manifold (i.e. a complex manifold endowed with a holomorphic action of $G$).
\begin{enumerate}[leftmargin=*]
\item $Y$ is a \emph{$G$-Oka manifold} if for any reductive closed subgroup $H$ of $G$ the fixed-point manifold $Y^{H}$ is Oka (cf. \cite[\S 2]{Kutzschebauchc}).
\item $Y$ satisfies \emph{Condition $G$-$\Ell_{1}$} if for any $G$-equivariant holomorphic map $f:X\to Y$ from a Stein $G$-manifold there exists a $G$-equivariant dominating global spray $s:E\to Y$ over $f$ where $E\to X$ is a holomorphic $G$-vector bundle.
\end{enumerate}
\end{definition}

\begin{theorem}
\label{theorem:G-Ell_1}
For any reductive complex Lie group $G$, a $G$-manifold $Y$ is $G$-Oka if it satisfies Condition $G$-$\Ell_{1}$.
\end{theorem}

\begin{proof}
Let $H$ be a reductive closed subgroup of $G$, $X$ be a Stein manifold and $f:X\to Y^{H}$ be a holomorphic map.
Consider the holomorphic action of $G$ on the Stein manifold $(G/H)\times X$ defined by $g\cdot([g'],x)=([gg'],x)$ and the $G$-equivariant holomorphic map $\tilde f:(G/H)\times X\to Y$, $([g],x)\mapsto g\cdot f(x)$.
By $G$-$\Ell_{1}$ of $Y$, there exists a $G$-equivariant dominating global spray $s:E\to Y$ over $\tilde f$.
Consider the holomorphic $H$-vector bundle $E_{0}=E|_{\{[1]\}\times X}$ over $X\cong \{[1]\}\times X$.
Then the fixed-point manifold $E_{0}^{H}$ is a holomorphic vector bundle over $X$ (cf. the  proof of \cite[Proposition 3.2]{Kutzschebauchc}) and the restriction $E_{0}^{H}\to Y^{H}$ of $s$ is a dominating global spray over $f:X\to Y^{H}$.
Therefore $Y^{H}$ satisfies $\Ell_{1}$ and thus it is Oka.
\end{proof}

In the case when $G$ is finite, Kutzschebauch, L\'{a}russon and Schwarz \cite[Corollary 4.2]{Kutzschebauchc} proved that the $G$-Oka property is equivalent to $G$-BOPAJI which easily implies Condition $G$-$\Ell_{1}$ (see also the proof of \cite[Corollary 4.3]{Kutzschebauchc}).
Thus we obtain the following characterization of $G$-Oka manifolds.

\begin{corollary}
\label{corollary:characterization}
For any finite group $G$, a $G$-manifold is $G$-Oka if and only if it satisfies Condition $G$-$\Ell_{1}$.
\end{corollary}

The following is the equivariant version of the localization principle.

\begin{theorem}
\label{theorem:equivariant_localization}
Let $G$ be a reductive complex Lie group and $Y$ be a $G$-manifold.
Assume that each point of $Y$ has a $G$-invariant open $G$-Oka neighborhood with respect to the analytic Zariski topology.
Then $Y$ is $G$-Oka.
\end{theorem}

\begin{proof}
Assume that $Y$ is covered by $G$-invariant open $G$-Oka subsets $U_{\lambda}$ $(\lambda\in\Lambda)$ with respect to the analytic Zariski topology.
Then for a reductive closed subgroup $H$ of $G$ the fixed-point manifold $Y^{H}$ is covered by Zariski open Oka subsets $U_{\lambda}^{H}$ $(\lambda\in\Lambda)$.
Thus $Y^{H}$ is Oka by the usual localization principle (Theorem \ref{theorem:localization}).
\end{proof}

The equivariant localization principle gives the following example of an equivariantly Oka manifold.

\begin{example}
\label{example:toric}
Let $Y$ be a smooth toric variety with the torus action $(\C^{*})^{n}\curvearrowright Y$.
Then $Y$ is covered by $(\C^{*})^{n}$-invariant Zariski open affine subsets of the form $\C^{k}\times(\C^{*})^{n-k}$.
Their fixed-point manifolds with respect to the action of a reductive closed subgroup of $(\C^{*})^{n}$ are also of the form $\C^{j}\times(\C^{*})^{l}$, and hence they are Oka.
Thus the smooth toric variety $Y$ is $(\C^{*})^{n}$-Oka.
\end{example}

%
%

\section*{Acknowledgement}
I wish to thank my supervisor Katsutoshi Yamanoi for helpful comments.
I also thank Franc Forstneri\v{c} and Finnur L\'{a}russon for valuable discussions at Ljubljana in 2018.
An elementary proof of that approximation implies interpolation was asked by L\'{a}russon at that time.
This work was supported by JSPS KAKENHI Grant Number JP18J20418.

%
%

\end{document}